\documentclass[oneside,english]{amsart}
\usepackage[T1]{fontenc}
\usepackage[latin9]{inputenc}
\usepackage{geometry}
\usepackage{amsthm}
\usepackage{amssymb}
\usepackage{esint}
\usepackage[all]{xy}

\usepackage{enumerate}
\newcommand{\path}[1]{\underline{#1}}

\makeatletter
\numberwithin{equation}{section}
\numberwithin{figure}{section}
  \theoremstyle{plain}
  \newtheorem*{thm*}{\protect\theoremname}
\theoremstyle{plain}
\newtheorem{thm}{\protect\theoremname}
  \theoremstyle{plain}
  \newtheorem{lem}[thm]{\protect\lemmaname}
  \theoremstyle{definition}
  \newtheorem{defn}[thm]{\protect\definitionname}
  \theoremstyle{plain}
  \newtheorem{prop}[thm]{\protect\propositionname}
  \theoremstyle{remark}
  \newtheorem{rem}[thm]{\protect\remarkname}

\usepackage{shuffle}

\makeatother

  \providecommand{\definitionname}{Definition}
  \providecommand{\lemmaname}{Lemma}
  \providecommand{\propositionname}{Proposition}
  \providecommand{\remarkname}{Remark}
  \providecommand{\theoremname}{Theorem}
\providecommand{\theoremname}{Theorem}

\begin{document}
\author{Minoru Hirose}
\address[Minoru Hirose]{Faculty of Mathematics, Kyushu University, 744, Motooka, Nishi-ku, Fukuoka, 819-0395, Japan}
\email{m-hirose@math.kyushu-u.ac.jp}
\author{Nobuo Sato}
\address[Nobuo Sato]{Faculty of Mathematics, Kyushu University, 744, Motooka, Nishi-ku, Fukuoka, 819-0395, Japan}
\email{n-sato@math.kyushu-u.ac.jp}

\title[Algebraic differential formulas of iterated integrals]{Algebraic differential formulas for the shuffle, stuffle and duality
relations of iterated integrals}

\date{\today}

\begin{abstract}
In this paper, we prove certain algebraic identities, which correspond
to differentiation of the shuffle relation, the stuffle relation,
and the relations which arise from M\"{o}bius transformations of
iterated integrals. These formulas provide fundamental and useful
tools in the study of iterated integrals on a punctured projective
line.
\end{abstract}
\maketitle

\section{Introduction}

Multiple zeta values (MZVs in short) are real numbers defined by the
convergent series
\[
\zeta(k_{1},\ldots,k_{d})=\sum_{0<m_{1}<\cdots<m_{d}}\frac{1}{m_{1}^{k_{1}}\cdots m_{d}^{k_{d}}}
\]
where $(k_{1},\ldots,k_{d})$ is a $d$-tuple of positive integers
such that $k_{d}>1$, and the sum is taken over all $d$-tuples of
positive integers satisfying the inequality. MZVs are, as obvious
from their definition, multiple series generalization of the Riemann
zeta function evaluated at positive integers.

Despite their simple and elementary appearance, the algebraic nature
of the space spanned by MZVs is surprisingly rich and mysterious.
In fact, MZVs arise in various areas of mathematics and particle physics,
such as in the Kontsevich integrals of knots (\cite{Le_Murakami_96},
\cite{Le_Murakami_95}, \cite{Le_Murakami_96_2}), or in the evaluation
of scattering amplitudes (\cite{amplitude_2013_1}, \cite{amplitude_2013_2}),
and are periods of mixed Tate motives over $\mathbb{Z}$ (\cite{Brown_MTMZ},\cite{Goncharov_Galois_symmetries}).

A particularly important aspect of MZVs is their iterated integral
representation due to Kontsevich, i.e.,
\begin{align*}
(-1)^{d}\zeta(k_{1},\dots,k_{d}) & =I(0;\overbrace{1,0,\ldots,0}^{k_{1}},\overbrace{1,0,\ldots,0}^{k_{2}},\ldots,\overbrace{1,0,\ldots,0}^{k_{d}};1)\\
 & \mbox{where } I(s;a_{1},\dots,a_{k};t):=\int_{s<t_{1}<\cdots<t_{k}<t}\frac{dt_{1}}{t_{1}-a_{1}}\wedge\cdots\wedge\frac{dt_{k}}{t_{k}-a_{k}},
\end{align*}
which (also occurs as the coefficients of the KZ-associator) provides
a geometric point of view on MZVs. To describe the algebraic structure
of the space spanned by such iterated integrals, Hoffman \cite{Hoffman_alg_setup}
introduced the non-commutative polynomial algebra generated by two
indeterminates $e_{0},e_{1}$ which correspond to the differential
one-forms $\frac{dt}{t}$ and $\frac{dt}{t-1}$, respectively. This
algebraic setup gives a fundamental tool to study MZVs. In \cite{HIST},
\cite{Confluence}, the authors consider more general iterated integrals
$I(0;a_{1},\dots,a_{k};1)$ with $a_{1},\dots,a_{k}\in\{0,1,z\},$
where $z$ is a complex variable. For this purpose, they extended
Hoffman's algebraic setup by adding another generator $e_{z}$ corresponding
to the one-form $\frac{dt}{t-z}$ to Hoffman's algebra.
More precisely, let $\mathcal{A}_{\{0,1,z\}}=\mathbb{Z}\left\langle e_{0},e_{1},e_{z}\right\rangle $
(resp. $\mathcal{A}_{\{0,1\}}=\mathbb{Z}\left\langle e_{0},e_{1}\right\rangle $)
be the non-commutative polynomial algebra over $\mathbb{Z}$ generated
by the indeterminates $e_{0},e_{1}$ and $e_{z}$ (resp. $e_{0}$
and $e_{1}$). Let $\mathcal{A}_{\{0,1,z\}}^{0}$ denote the subspace
of admissible elements
\[
\mathcal{A}_{\{0,1,z\}}^{0}:=\mathbb{Z}\oplus\mathbb{Z}e_{z}\oplus\bigoplus_{\substack{a\in\{1,z\}\\
b\in\{0,z\}
}
}e_{a}\mathcal{A}_{\{0,1,z\}}e_{b},
\]
$\mathcal{A}_{\{0,1\}}^{0}$ its subspace
\begin{align*}
\mathcal{A}_{\{0,1\}}^{0} & =\mathcal{A}_{\{0,1,z\}}^{0}\cap\mathcal{A}_{\{0,1\}}=\mathbb{Z}\oplus e_{1}\mathcal{A}_{\{0,1\}}e_{0}
\end{align*}
and $L$ a linear map on $\mathcal{A}_{\{0,1,z\}}^{0}$ defined by
\[
L(e_{a_{1}}\cdots e_{a_{n}})=\int_{0<t_{1}<\cdots<t_{n}<1}\prod_{j=1}^{n}\frac{dt_{j}}{t_{j}-a_{j}}.
\]
Then $L$ satisfies the following differential formula \cite{HIST}
\begin{equation}
\frac{d}{dz}L(w)=\sum_{c\in\{0,1\}}\frac{1}{z-c}L(\partial_{z,c}w),\label{eq:der_intro1}
\end{equation}
where $\partial_{\alpha,\beta}:\mathcal{A}_{\{0,1,z\}}\to\mathcal{A}_{\{0,1,z\}}$
is a linear map defined by
\begin{equation}
\partial_{\alpha,\beta}(e_{a_{1}}\cdots e_{a_{n}})=\sum_{i=1}^{n}\left(\delta_{\{a_{i},a_{i+1}\},\{\alpha,\beta\}}-\delta_{\{a_{i-1},a_{i}\},\{\alpha,\beta\}}\right)e_{a_{1}}\cdots\widehat{e_{a_{i}}}\cdots e_{a_{n}}\label{eq:der_intro}
\end{equation}
with $a_{0}=0$, $a_{n+1}=1$. Here, $\delta_{S,T}$ is the Kronecker
delta for two sets $S$ and $T$.
These $\partial_{\alpha,\beta}$'s are the algebraic counterpart
of the usual differentiation $d/dz$, and have fundamental importance
in the study of iterated integrals. For example, in \cite{HIST},
the authors proved a ``sum formula'' for iterated integrals, which
generalizes the classical sum formula for MZVs, by using its inductive
structure with respect to the algebraic differentiation. Also, in
\cite{Confluence}, the authors exploited the differential structure
of $\mathcal{A}_{\{0,1,z\}}^{0}$ to construct a class of (presumably
exhausting all the) relations among MZVs, which they called the confluence
relation. The purpose of this paper to investigate the relationships
between $\partial_{\alpha,\beta}$ and other basic algebraic operations.

We denote the shuffle product, the
stuffle product and the duality map by $\shuffle$, $*$ and $\tau_{z}$,
respectively (here, the precise definition of $*$ and $\tau_{z}$
will be explained later). Then, $L$ satisfies the ``shuffle relation''
\begin{equation}
L(u\shuffle v)=L(u)L(v),\label{eq:shuffle}
\end{equation}
the ``stuffle relation'' \cite[Section 5.2]{BBBL_stuffle}
\begin{equation}
L(u*v)=L(u)L(v),\label{eq:stuffle}
\end{equation}
and the ``duality relation'' \cite[Theorem 1.1]{HIST}
\begin{equation}
L(\tau_{z}(u))=L(u).\label{eq:duality}
\end{equation}
By differentiating the equalities (\ref{eq:shuffle}), (\ref{eq:stuffle})
and (\ref{eq:duality}) with respect to $z$ and applying (\ref{eq:der_intro1}),
we obtain
\begin{align}
\sum_{c\in\{0,1\}}\frac{1}{z-c}L(\partial_{z,c}(u\shuffle v)) & =\sum_{c\in\{0,1\}}\frac{1}{z-c}(L(\partial_{z,c}u)L(v)+L(u)L(\partial_{z,c}v)),\label{eq:evalated_by_L}\\
\sum_{c\in\{0,1\}}\frac{1}{z-c}L(\partial_{z,c}(u*v)) & =\sum_{c\in\{0,1\}}\frac{1}{z-c}L(u)L(\partial_{z,c}v),\\
\sum_{c\in\{0,1\}}\frac{1}{z-c}L(\partial_{z,c}\tau_{z}(u)) & =\sum_{c\in\{0,1\}}\frac{1}{z-c}L(\tau_{z}(\partial_{z,c}u)).
\end{align}
Therefore, it is natural to ask whether these equalities in complex
numbers lift to the equalities in $\mathcal{A}_{\{0,1,z\}}$.
In this paper we shall show that the answers to all these three questions are ``Yes''.
More precisely, we shall prove the following theorem.
\begin{thm*}
For $c\in\{0,1\}$,
\begin{align*}
\partial_{z,c}(u\shuffle v) & =(\partial_{z,c}u)\shuffle v+u\shuffle(\partial_{z,c}v)\qquad\left(u,v\in\mathcal{A}_{\{0,1,z\}}^{0}\right)\\
\partial_{z,c}(u*v) & =u*(\partial_{z,c}v)\qquad\qquad\;\,\left(u\in\mathcal{A}_{\{0,1\}}^{0},\, v\in\mathcal{A}_{\{0,1,z\}}^{0}\right)\\
\partial_{z,c}\tau_{z}(u) & =\tau_{z}(\partial_{z,c}u)\qquad\qquad\qquad\qquad\qquad\left(u\in\mathcal{A}_{\{0,1,z\}}^{0}\right).
\end{align*}

\end{thm*}
The first two equalities state that the linear operator $\partial_{z,c}$
is a derivation with respect to both the shuffle and the stuffle products
(the second equality can also be written as $\partial_{z,c}(u*v)=(\partial_{z,c}u)*v+u*(\partial_{z,c}v)$
since $\partial_{z,c}u=0$ for $u\in\mathcal{A}_{\{0,1\}}^{0}$).

These formulas provide fundamental and useful tools in the study of MZVs and iterated integrals. Let us discuss some
significance of them. First of all, the formulas are purely algebraic
and therefore have wide applications. One of such important applications
is given in \cite{Confluence}, where the authors have proved that the confluence
relation implies the regularized double shuffle and the duality relations.
Recently, the confluence relation was proved to be equivalent to the
associator relation by Furusho \cite{Furusho_equivalence_of_AS_and_CF}, and so this result can also be viewed
as an alternative proof of the main result of \cite{Furusho_AS_implpies_RDS}. Furthermore, let $\gamma$ be a general path from $0$ to $1$,
and $L_{\gamma}$ denote the map similarly defined as $L$, where
the defining iterated integral is replaced with that along $\gamma$.
By applying $L_{\gamma}$ to the algebraic differential formula of
the stuffle product, one can obtain
\begin{equation}
\frac{d}{dz}L_{\gamma}(u*v)=\sum_{c\in\{0,1\}}\frac{1}{z-c}L_{\gamma}(u*\partial_{z,c}v).\label{eq:diff_stuffle_general_path}
\end{equation}
Note that for the trivial path $\gamma(t)=t$, this identity gives
an alternative proof of the stuffle relation (\ref{eq:stuffle}). Thus, we may naturally
expect that (\ref{eq:diff_stuffle_general_path}) would serve as a
first step to discover and prove the counterpart of (\ref{eq:stuffle}) for $L_{\gamma}$ with more general $\gamma$.
We also define a natural generalization of the stuffle product and
prove the three formulas above in far more general forms (see Theorems
 \ref{thm:shuffle}, \ref{thm:stuffle}, \ref{thm:lin_fra_main}).
It is quite likely that these formulas have wide applications
in general iterated integrals, for instance, for proving analogous
results as in \cite{Confluence} (i.e., the implication of the regularized double
shuffle and the duality relations by the confluence relation) for
Euler sums, multiple L-values and even more general iterated integrals.

This paper is organized as follows. In Section \ref{sec:Basic-settings}, we introduce some basic settings
and state a useful lemma used in the proof of the theorem. In Section
\ref{sec:Shuffle}, we prove the algebraic differential formula for
the shuffle product (Theorem \ref{thm:shuffle}). In Section \ref{sec:Stuffle},
we prove the algebraic differential formula for the stuffle product (Theorem
\ref{thm:stuffle}). In Section \ref{sec:Mbius-transformation}, we
introduce algebraic M\"{o}bius transformations as a generalization
of the duality map, and prove their algebraic differential formula
(Theorem \ref{thm:lin_fra_main}). In Section \ref{sec:An-application-to01z},
we derive the theorem above as special cases of Theorems \ref{thm:shuffle},
\ref{sec:Stuffle}, \ref{sec:Mbius-transformation}.

\section{\label{sec:Basic-settings}Basic settings}

Let $F$ be a field and $\mathcal{A}=\mathcal{A}_{F}$ be the non-commutative
free algebra over $\mathbb{Z}$ generated by the indeterminates $\{e_{p}\mid p\in F\}$.
For $s,t\in F$, we denote by $\mathcal{A}_{(s,t)}^{0}=\mathcal{A}_{F,(s,t)}^{0}$
the subalgebra
\[
\mathcal{A}_{F,(s,t)}^{0}=\mathbb{Z}\oplus\bigoplus_{z\in F\setminus\{s,t\}}\mathbb{Z}e_{z}\oplus\bigoplus_{\substack{x\in F\setminus\{s\}\\
y\in F\setminus\{t\}
}
}e_{x}\mathcal{A}_{F}e_{y}\subset\mathcal{A}_{F}.
\]
In particular, we put $\mathcal{A}_{F}^{0}=\mathcal{A}_{F,(0,1)}^{0}$.

We fix a homomorphism $\mathfrak{F}:F^{\times}\to\mathbb{Z}$. For
$x\in F$, define $[x]\in\mathbb{Z}$ by
\[
[x]=\begin{cases}
\mathfrak{F}(x) & x\neq0\\
0 & x=0.
\end{cases}
\]
Note that since $\mathbb{Z}$ is torsion-free, $[-x]=[x]$ for $x\in F^{\times}$.
We define an algebraic differential operator $\partial^{s,t}=\partial_{\mathfrak{F}}^{s,t}:\mathcal{A}\to\mathcal{A}$
by%
\footnote{The motivation to consider such an operator comes from the following
differential formula (Goncharov \cite[Theorem 2.1]{GonTotdif}, Panzer
\cite[Lemma 3.3.30]{PanzerArxiv})
\begin{align*}
dI(s;a_{1},\dots,a_{n};t)= & \sum_{\substack{1\leq i\leq n\\
a_{i}\neq a_{i+1}
}
}I(s;a_{1},\dots,\widehat{a_{i}},\dots,a_{n};t)\, d\log(a_{i+1}-a_{i})\\
 & -\sum_{\substack{1\leq i\leq n\\
a_{i}\neq a_{i-1}
}
}I(s;a_{1},\dots,\widehat{a_{i}},\dots,a_{n};t)\, d\log(a_{i}-a_{i-1})\quad\left((a_{0},a_{n+1})=(s,t)\right),
\end{align*}
where $I$ is the iterated integral symbol by Goncharov \cite{GonTotdif}.%
}
\[
\partial^{s,t}(e_{a_{1}}\cdots e_{a_{n}}):=\sum_{i=1}^{n}([a_{i+1}-a_{i}]-[a_{i}-a_{i-1}])e_{a_{1}}\cdots\widehat{e_{a_{i}}}\cdots e_{a_{n}}\qquad\left((a_{0},a_{n+1})=(s,t)\right).
\]
In particular, we put $\partial=\partial^{0,1}$. The following lemma
provides a useful technique in later sections.
\begin{lem}
\label{lem:lift_ad}Let $a_{0},\dots,a_{n+1}\in F$ (with $(a_{0},a_{n+1})=(s,t)$)
and $f:\{0,\dots,n\}\to\mathbb{Z}$ be any map such that
\[
f(i)=[a_{i+1}-a_{i}]\qquad\mbox{for all }i\in\{0\leq j\leq n\mid a_{j}\neq a_{j+1}\}.
\]
Then, we have
\begin{align*}
\partial^{s,t}(e_{a_{1}}\cdots e_{a_{n}})= & \sum_{i=1}^{n}\left(f(i)-f(i-1)\right)e_{a_{1}}\cdots\widehat{e_{a_{i}}}\cdots e_{a_{n}}\\
 & +\delta_{s,a_{1}}f(0)\, e_{a_{2}}\cdots e_{a_{n}}-\delta_{a_{n},t}f(n)\, e_{a_{1}}\cdots e_{a_{n-1}}.
\end{align*}

\end{lem}
The proof follows directly from the definition of $\partial^{s,t}$. Note that
the right-hand side in the above equality does not depend on the choice
of $f$ and so we choose suitable $f$ which is convenient for the
situation.

As a general notation, for a given function $f:S\rightarrow\mathcal{A}$
on a finite set $S$, we define
\[
f(T):=\sum_{x\in T}f(x)
\]
for $T\subset S$. We often use this notation in the following sections.

\section{\label{sec:Shuffle}Differential formula for the shuffle product}

In this section, we shall prove the derivation property of the algebraic
differential operator with respect to the shuffle product.
\begin{defn}
We define the shuffle product $\shuffle:\mathcal{A}\times\mathcal{A}\to\mathcal{A}$
inductively by $w\shuffle1=1\shuffle w=w$ for $w\in\mathcal{A}$
and
\[
e_{a}u\shuffle e_{b}v=e_{a}(u\shuffle e_{b}v)+e_{b}(e_{a}u\shuffle v)\qquad(u,v\in\mathcal{A}).
\]

\end{defn}
As is well known, the shuffle product can also be defined in a more
combinatorial way as
\[
a_{1}\cdots a_{n}\,\shuffle\, a_{n+1}\cdots a_{n+m}=\sum_{\sigma\in S(n,m)}a_{\sigma(1)}\cdots a_{\sigma(n+m)}\qquad(a_{1},\ldots,a_{n+m}\in\{e_{p}\mid p\in F\}),
\]
where $S(n,m)$ is a set of all permutations $\sigma$'s of $\{1,2,\dots,n+m\}$
such that $\sigma^{-1}(1)<\sigma^{-1}(2)<\cdots<\sigma^{-1}(n)$ and
$\sigma^{-1}(n+1)<\sigma^{-1}(n+2)<\cdots<\sigma^{-1}(n+m)$.
\begin{thm}
\label{thm:shuffle}For $u,v\in\mathcal{A}$, $\partial(u\shuffle v)=(\partial u)\shuffle v+u\shuffle(\partial v)$.\end{thm}
\begin{proof}
It is sufficient to show the equality for monomials $u$ and $v$.
Put $u=a_{1}\cdots a_{n}$, $v=a_{n+1}\cdots a_{n+m}$ ($a_{1},\ldots,a_{n+m}\in\{e_{p}\mid p\in F\}$)
and $X=S(n,m)\times\{0,1,\dots,n+m\}$. Then, by definition,
\[
\partial(u\shuffle v)=f(X):=\sum_{(\sigma,i)\in X}f(\sigma,i),
\]
where $f(\sigma,i):=f^{+}(\sigma,i)-f^{-}(\sigma,i)$ with
\begin{align*}
f^{+}(\sigma,i):= & \begin{cases}
a_{\sigma(1)}\cdots\widehat{a_{\sigma(i)}}\cdots a_{\sigma(n+m)}[a_{\sigma(i+1)}-a_{\sigma(i)}] & (i\neq0)\\
0 & (i=0),
\end{cases}\\
f^{-}(\sigma,i):= & \begin{cases}
a_{\sigma(1)}\cdots\widehat{a_{\sigma(i+1)}}\cdots a_{\sigma(n+m)}[a_{\sigma(i+1)}-a_{\sigma(i)}] & (i\neq n+m)\\
0 & (i=n+m).
\end{cases}
\end{align*}
Put $I:=\{1,\dots,n\},$ $J:=\{n+1,\dots,n+m\}$ and
\[
X_{AB}:=\{(\sigma,i)\in X\mid0<i<n+m,\,\sigma(i)\in A,\,\sigma(i+1)\in B\}
\]
for $A,B\in\{I,J\}$. Further, we put
\begin{align*}
\widetilde{X_{II}}: & =X_{II}\sqcup\left\{ (\sigma,0)\in X\mid\sigma(1)\in I\right\} \sqcup\left\{ (\sigma,n+m)\in X\mid\sigma(n+m)\in I\right\} \\
\widetilde{X_{JJ}}: & =X_{JJ}\sqcup\left\{ (\sigma,0)\in X\mid\sigma(1)\in J\right\} \sqcup\left\{ (\sigma,n+m)\in X\mid\sigma(n+m)\in J\right\} .
\end{align*}
Then we have
\[
X=X_{IJ}\sqcup X_{JI}\sqcup\widetilde{X_{II}}\sqcup\widetilde{X_{JJ}}.
\]
Define a bijection $\tau:X_{IJ}\to X_{JI}$ by $\tau(\sigma,i)=(\sigma\circ\tau_{i},i)$,
where $\tau_{i}$ denotes the transposition of $i$ and $i+1$. Since
$f^{\pm}\circ\tau=f^{\mp}$, we have $f\circ\tau=-f$. Thus
\[
f(X_{IJ})+f(X_{JI})=0.
\]
Since
\begin{align*}
f(\widetilde{X_{II}}) & =(\partial u)\shuffle v,\\
f(\widetilde{X_{JJ}}) & =u\shuffle(\partial v)
\end{align*}
by definition, it readily follows that
\[
\partial(u\shuffle v)=(\partial u)\shuffle v+u\shuffle(\partial v).
\]

\end{proof}

\section{\label{sec:Stuffle}Differential formula for the stuffle product}

In this section, we shall prove the derivation property of the algebraic
differential operator with respect to the stuffle product.

\subsection{Definition of the generalized stuffle product}
\begin{defn}
\label{Stuffle product}We define the (generalized) stuffle product
$*:\mathcal{A}\times\mathcal{A}\to\mathcal{A}$ inductively by $w*1=1*w=w$
and
\begin{align*}
e_{a}u*e_{b}v & =e_{ab}(u*e_{b}v+e_{a}u*v-e_{0}(u*v))\,\qquad(u,v\in\mathcal{A}).
\end{align*}

\end{defn}
Note that in our definition of the stuffle product, $e_{0}$ plays
a particular role. Thus, we check the compatibility of our definition
with the usual one (\cite[Section 5]{BBBL_stuffle}, \cite[Section 1]{IKZ})
in the first place. Let $\mathfrak{h}^{1}$ be the non-commutative
free algebra generated by infinitely many formal indeterminates $\{z_{k,a}\mid k\in\mathbb{Z}_{\geq1},a\in F^{\times}\}$.
Define a binary operator $\bar{*}:\mathfrak{h}^{1}\times\mathfrak{h}^{1}\to\mathfrak{h}^{1}$
by $1\bar{*}w=w\bar{*}1=w$ for $w\in\mathfrak{h}^{1}$ and
\[
z_{k,a}u\bar{*}z_{l,b}v=z_{k,ab}(u\bar{*}z_{l,b}v)+z_{l,ab}(z_{k,a}u\bar{*}v)+z_{k+l,ab}(u\bar{*}v)\quad(u,v\in\mathfrak{h}^{1}).
\]
We define an embedding $i:\mathfrak{h}^{1}\hookrightarrow\mathcal{A}$
of algebra by $i(uv)=i(u)i(v)$ and $i(z_{k,a})=-e_{a}e_{0}^{k-1}$.
The following proposition assures that our stuffle product is a generalization
of the usual stuffle products.
\begin{prop}
\
\begin{enumerate}[\upshape (i)]
\item For $u,v\in\mathfrak{h}^{1}$, $i(u\bar{*}v)=i(u)*i(v)$.
\item The stuffle product $*$ is commutative and associative.
\end{enumerate}
\end{prop}
\begin{proof}[Proof of {\upshape(i)}]
It is sufficient to show the equality for monomials $u$ and $v$.
First of all, one can easily check that
\[
e_{0}u*v=e_{0}(u*v)
\]
holds for any monomials $u,v$ by induction on the degree of $v$.
Now we shall prove the proposition by induction. Set $u=z_{k,a}u'$,
$v=z_{l,b}v'$. Then we get
\begin{align*}
i(u)*i(v)= & (e_{a}e_{0}^{k-1}i(u'))*(e_{b}e_{0}^{l-1}i(v'))\\
= & e_{ab}(e_{0}^{k-1}i(u')*e_{b}e_{0}^{l-1}i(v'))+e_{ab}(e_{a}e_{0}^{k-1}i(u')*e_{0}^{l-1}i(v'))\\
 & -e_{ab}e_{0}(e_{0}^{k-1}i(u')*e_{0}^{l-1}i(v'))\\
= & -e_{ab}e_{0}^{k-1}(i(u')*i(v))-e_{ab}e_{0}^{l-1}(i(u)*i(v'))\\
 & -e_{ab}e_{0}^{k+l-1}(i(u')*i(v')).
\end{align*}
Using the induction hypothesis, the last quantity is equal to
\begin{align*}
 & -e_{ab}e_{0}^{k-1}i(u'\bar{*}v)-e_{ab}e_{0}^{l-1}i(u\bar{*}v')-e_{ab}e_{0}^{k+l-1}i(u'\bar{*}v')\\
 & =i(z_{k,ab}(u'\bar{*}v)+z_{l,ab}(u\bar{*}v')+z_{k+l,ab}(u'\bar{*}v'))\\
 & =i(u\bar{*}v)
\end{align*}
which proves the proposition.
\end{proof}
\begin{proof}[Proof of {\upshape(ii)}]
The commutativity of the stuffle product is obvious from the definition.
Let us prove the associativity. We define a trilinear map $f:\mathcal{A}\times\mathcal{A}\times\mathcal{A}\to\mathcal{A}$
inductively by
\[
f(u,v,1)=f(u,1,v)=f(1,u,v)=u*v
\]
and
\begin{align*}
f(u,v,w):= & e_{abc}\left(f(u',v,w)+f(u,v',w)+f(u,v,w')\right)\\
 & -e_{abc}e_{0}\left(f(u',v',w)+f(u',v,w')+f(u,v',w')\right)\\
 & +e_{abc}e_{0}^{2}f(u',v',w')\ \ \ \ (u=e_{a}u',v=e_{b}v',w=e_{c}w').
\end{align*}
Then we can show $(u*v)*w=f(u,v,w)$ by the induction since
\begin{align*}
(u*v)*w= & e_{ab}(u*v'+u'*v-e_{0}(u'*v'))*w\\
= & e_{abc}((u*v)*w')+e_{abc}\left((u*v'+u'*v-e_{0}(u'*v'))*w\right)\\
 & -e_{abc}e_{0}((u*v'+u'*v-e_{0}(u'*v'))*w')\\
= & e_{abc}((u*v)*w'+(u*v')*w+(u'*v)*w)\\
 & -e_{abc}e_{0}((u'*v')*w+(u*v')*w'+(u'*v)*w')\\
 & +e_{abc}e_{0}^{2}(u'*v'*w')\ \ \ \ (u=e_{a}u',v=e_{b}v',w=e_{c}w').
\end{align*}
Since $f(u,v,w)$ is a symmetric function by definition, we
have
\[
(u*v)*w=f(u,v,w)=f(v,w,u)=(v*w)*u=u*(v*w)
\]
for $u,v,w\in\mathcal{A}$. This is the associativity of $*$.
\end{proof}

\subsection{Combinatorial description of the stuffle product.\label{sub:CombSt}}

The stuffle product as well as the shuffle product enjoys an interesting
combinatorial structure. In this section, we give such a description
of the stuffle product. This description will be useful in the proof
of the algebraic differential formula stated in the next section.

Fix positive integers $m$ and $n$.
Let $G=(V,E)$ be a directed graph whose vertex set $V$ and edge
set $E$ are given by
\[
V=\left\{ \left.(x,y)\in\frac{1}{2}\mathbb{Z}\times\frac{1}{2}\mathbb{Z}\right|\begin{array}{c}
1\leq x\leq n+1,\\
\,1\leq y\leq m+1,
\end{array}\, x-y\in\mathbb{Z}\right\} \sqcup\left\{ \left(\frac{1}{2},\frac{1}{2}\right)\right\}
\]
and by
\begin{align*}
E= & \:\{\left((x,y),(x+1,y)\right)\in V^{2}\mid x,y\in\mathbb{Z}\}\\
 & \sqcup\{\left((x,y),(x,y+1)\right)\in V^{2}\mid x,y\in\mathbb{Z}\}\\
 & \sqcup\left\{ \left.\left((x,y),(x+\frac{1}{2},y+\frac{1}{2})\right)\in V^{2}\right|x,y\in\frac{1}{2}\mathbb{Z}\right\} ,
\end{align*}
respectively (see Figure \ref{fig:Example}).
\begin{figure}[h]
\[
\xymatrix{ & (2,1)\ar[rr] &  & (2,2)\ar[rr] &  & (2,3)\\
 &  & (1\frac{1}{2},1\frac{1}{2})\ar[ru] &  & (1\frac{1}{2},2\frac{1}{2})\ar[ru]\\
 & (1,1)\ar[rr]\ar[ru]\ar[uu] &  & (1,2)\ar[rr]\ar[ru]\ar[uu] &  & (1,3)\ar[uu]\\
(\frac{1}{2},\frac{1}{2})\ar[ru]
}
\]
\caption{\label{fig:Example}The graph $G$ for the case $(n,m)=(1,2)$}
\end{figure}
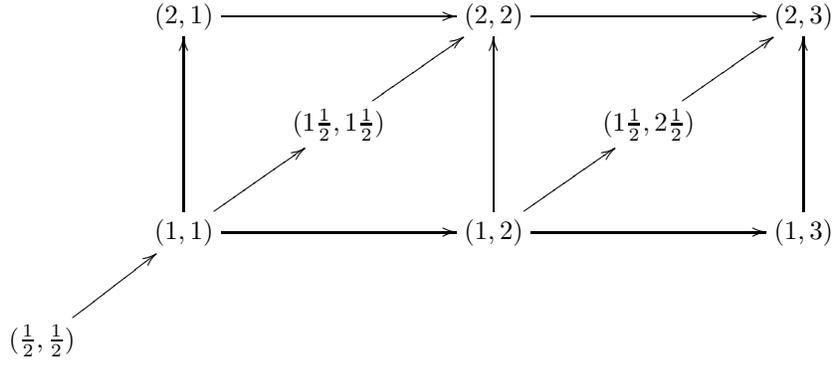
We denote by $P$ the set of paths from $(1/2,1/2)$ to $(n+1,m+1)$,
i.e.,
\begin{align*}
P: & =\left\{ \path{p}=\left(p_{0},p_{1},\ldots,p_{n+m+1}\right)\left|\,\begin{array}{c}
p_{0}=\left(1/2,1/2\right),\\
p_{n+m+1}=(n+1,m+1),
\end{array}\begin{array}{c}
(p_{i},p_{i+1})\in E\\
\mbox{ for }0\leq i\leq n+m
\end{array}\right.\right\}.
\end{align*}
Hereafter, for a path $\path{p}\in{P}$, we denote by $p_i$ the $(i+1)$-th entry of $\path{p}$, i.e.,
$\path{p}=\left(p_{0},p_{1},\ldots,p_{n+m+1}\right)$.
Furthermore we define ${\rm sgn}:P\to\{\pm1\}$ by
\[
{\rm sgn}\left(\path{p}\right)=\prod_{\substack{1\leq i\leq n+m\\
p_{i}\notin\mathbb{Z}^{2}
}
}(-1).
\]
For monomials $u=e_{a_{1}}\cdots e_{a_{n}}$ and $v=e_{b_{1}}\cdots e_{b_{m}}$ of degree $n$ and $m$ respectively, define $f_{u,v}:V\to F$ by
\[
f_{u,v}(x,y)=\begin{cases}
a_{x}b_{y} & (x,y)\in\mathbb{Z}^{2}\\
0 & (x,y)\notin\mathbb{Z}^{2}.
\end{cases}
\]
Here, we set $a_{n+1}=b_{m+1}=1$. Then, we have the
following combinatorial expression for the stuffle product.
\begin{prop}
\label{prop:Stuffle product}
For monomials $u$ and $v$ of degree $n$ and $m$ respectively,
\[
u*v=\sum_{\path{p}\in P}{\rm sgn}(\path{p})\, e_{f_{u,v}(p_{1})}\cdots e_{f_{u,v}(p_{n+m})}
\]

\end{prop}
One can easily check that this expression is equivalent to Definition
\ref{Stuffle product}.
\begin{rem}
The vertices $p_{0}=(\frac{1}{2},\frac{1}{2})$ and $p_{n+m+1}=(n+1,m+1)$
do not appear in Proposition \ref{prop:Stuffle product}, but it is
convenient in the description of the proof of Theorem \ref{thm:stuffle}.
\end{rem}

\subsection{An algebraic differential formula of the stuffle product.}

Put
\[
\mathcal{A}^{1}=\mathcal{A}_{F}^{1}=\mathbb{Z}\oplus\bigoplus_{z\in F^{\times}}e_{z}\mathcal{A}_{F}\subset\mathcal{A}_{F}
\]
(note that $\mathcal{A}^{0}\subset\mathcal{A}^{1}$). In this section,
we prove the following identity:
\begin{thm}
\label{thm:stuffle}For non-constant monomials $u,v\in\mathcal{A}$,
we have
\[
\partial(u*v)=(\partial u)*v+u*\partial v+\delta_{a,0}[b](u'*v)+\delta_{b,0}[a](u*v'),
\]
where $u=e_{a}u',v=e_{b}v'$. In particular, for $u,v\in\mathcal{A}^{1}$,
\[
\partial(u*v)=(\partial u)*v+u*\partial v.
\]
\end{thm}
\begin{proof}
It is sufficient to show the identity for non-constant monomials $u$
and $v$. Put $u=e_{a_{1}}\cdots e_{a_{n}}$ and $v=e_{b_{1}}\cdots e_{b_{m}}$.
Let $P$ be the set of paths as defined in Section \ref{sub:CombSt}.
Put $f=f_{u,v}$ and $Q=P\times\{1,\cdots,n+m\}$.
Define $h:E\to\mathbb{Z}$ by
\[
h(p,p')=\begin{cases}
[a_{x+1}-a_{x}]+[b_{y}] & \mbox{ if }(p,p')=((x,y),(x+1,y))\\
{}[b_{y+1}-b_{y}]+[a_{x}] & \mbox{ if }(p,p')=((x,y),(x,y+1))\\
{}[a_{x}]+[b_{y}] & \mbox{ if }(p,p')=\begin{cases}
((x,y),(x+\frac{1}{2},y+\frac{1}{2}))\\
\mbox{ or }((x-\frac{1}{2},y-\frac{1}{2}),(x,y))
\end{cases}
\end{cases}
\]
for $x,y\in\mathbb{Z}$. Then $[f(p)-f(p')]=h(p,p')$ for all $(p,p')\in E$
such that $f(p)\neq f(p')$ (see Figure \ref{fig:ValuesFH}).

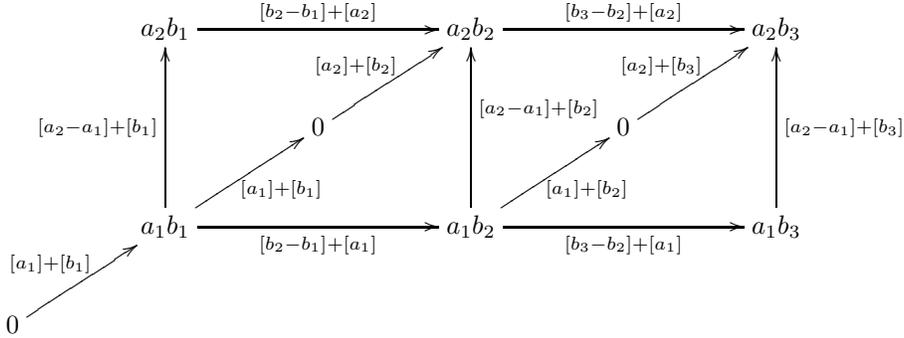
\begin{figure}[h]
\[
\xymatrix@C=40pt{ & a_{2}b_{1}\ar^-{[b_{2}-b_{1}]+[a_{2}]}[rr] &  & a_{2}b_{2}\ar^-{[b_{3}-b_{2}]+[a_{2}]}[rr] &  & a_{2}b_{3}\\
 &  & 0\ar^-{[a_{2}]+[b_{2}]}[ru] &  & 0\ar^-{[a_{2}]+[b_{3}]}[ru]\\
 & a_{1}b_{1}\ar_-{[b_{2}-b_{1}]+[a_{1}]}[rr]\ar_-{[a_{1}]+[b_{1}]}[ru]\ar^-{[a_{2}-a_{1}]+[b_{1}]}[uu] &  & a_{1}b_{2}\ar_-{[b_{3}-b_{2}]+[a_{1}]}[rr]\ar_-{[a_{1}]+[b_{2}]}[ru]\ar_(.6){[a_{2}-a_{1}]+[b_{2}]}[uu] &  & a_{1}b_{3}\ar_-{[a_{2}-a_{1}]+[b_{3}]}[uu]\\
0\ar^-{[a_{1}]+[b_{1}]}[ru]
}
\]
\caption{\label{fig:ValuesFH}Values of $f:V\to\mathbb{Z}$ and $h:E\to\mathbb{Z}$ }
\end{figure}

We define $s:Q\to\mathcal{A}_{F}$ by
\[
s(\path{p},i)={\rm sgn}(\path{p})\left(h(p_{i},p_{i+1})-h(p_{i-1},p_{i})\right)e_{f(p_{1})}\cdots\widehat{e_{f(p_{i})}}\cdots e_{f(p_{n+m})}
\]
for $\path{p}\in P$ and $i\in\{1,\dots,n+m\}$. Then, by Lemma \ref{lem:lift_ad},
we have
\begin{align*}
\partial(u*v)= & s(Q)+\sum_{\path{p}\in P}\delta_{f(p_{0}),f(p_{1})}h(p_{0},p_{1})\mathrm{sgn}(\path{p})e_{f(p_{2})}\cdots e_{f(p_{n+m})}\\
 & -\sum_{\path{p}\in P}\delta_{f(p_{n+m}),f(p_{n+m+1})}h(p_{n+m},p_{n+m+1})\mathrm{sgn}(\path{p})e_{f(p_{1})}\cdots e_{f(p_{n+m-1})},
\end{align*}
where $s(Q):=\sum_{(\path{p},i)\in Q}s(\path{p},i)$. Since
\[
\delta_{f(p_{0}),f(p_{1})}h(p_{0},p_{1})=\delta_{0,a_{1}b_{1}}h((\frac{1}{2},\frac{1}{2}),(1,1))=\delta_{0,a_{1}b_{1}}([a_{1}]+[b_{1}])
\]
and
\begin{align*}
\delta_{f(p_{n+m}),f(p_{n+m+1})}h(p_{n+m},p_{n+m+1}) & =\begin{cases}
\delta_{a_{n},1}([1-a_{n}]+[b_{m+1}]) & \mbox{ if }p_{n+m}=(n,m+1)\\
\delta_{b_{m},1}([1-b_{m}]+[a_{n+1}]) & \mbox{ if }p_{n+m}=(n+1,m)\\
0 & \mbox{ if }p_{n+m}=(n+\frac{1}{2},m+\frac{1}{2})
\end{cases}\\
 & =0,
\end{align*}
we have
\begin{equation}
\partial(u*v)=s(Q)+\Lambda,\label{eq:stpr0}
\end{equation}
where
\[
\Lambda:=\delta_{a_{1}b_{1},0}([a_{1}]+[b_{1}])\sum_{\path{p}\in P}{\rm sgn}(\path{p})e_{f(p_{2})}\cdots e_{f(p_{n+m})}.
\]
Since $\delta_{ab,0}[a]=\delta_{b,0}[a]$,
\begin{equation}
\Lambda=\delta_{a_{1},0}[b_{1}](u'*v)+\delta_{b_{1},0}[a_{1}](u*v'),\label{eq:lambda_value}
\end{equation}
where $u'=e_{a_{2}}\cdots e_{a_{n}}$ and $v'=e_{b_{2}}\cdots e_{b_{m}}$.
We decompose $Q$ as $Q=Q_{\mathbb{Z}}\sqcup Q_{\#}$, where
\begin{align*}
Q_{\mathbb{Z}}: & =\{(\path{p},i)\in Q\mid p_{i}\in\mathbb{Z}^{2}\},\\
Q_{\#}: & =\{(\path{p},i)\in Q\mid p_{i}\notin\mathbb{Z}^{2}\}.
\end{align*}
Moreover, by putting
\[
V_{\alpha}:=\begin{cases}
(1,0) & \mbox{ if }\alpha=\uparrow\\
(\frac{1}{2},\frac{1}{2}) & \mbox{ if }\alpha=\nearrow\\
(0,1) & \mbox{ if }\alpha=\rightarrow
\end{cases}
\]
and defining the sets
\[
 Q_{\alpha,\beta} :=\{(\path{p},i)\in Q\mid i>0,p_{i}\in\mathbb{Z}^{2},\, p_{i}=p_{i-1}+V_{\alpha},\, p_{i+1}=p_{i}+V_{\beta}\}\\
\]
for $\alpha,\beta,\gamma\in\{\uparrow,\nearrow,\rightarrow\}$, we
can further decompose $Q_{\mathbb{Z}}$ as
\[
Q_{\mathbb{Z}} =\bigsqcup_{\alpha,\beta\in\{\uparrow,\nearrow,\rightarrow\}}Q_{\alpha,\beta}.\\
\]
Now we have
\begin{align}
s(Q_{\nearrow\nearrow}) & =0\label{eq:ne_ne_zero}\\
s(Q_{\uparrow\rightarrow})+s(Q_{\rightarrow\uparrow})+s(Q_{\#}) & =0\label{eq:sharp_zero}\\
s(Q_{\rightarrow\rightarrow})+s(Q_{\rightarrow\nearrow})+s(Q_{\nearrow\rightarrow}) & =u*\partial v\label{eq:ee_der}\\
s(Q_{\uparrow\uparrow})+s(Q_{\uparrow\nearrow})+s(Q_{\nearrow\uparrow}) & =(\partial u)*v.\label{eq:nn_der}
\end{align}

The identity (\ref{eq:ne_ne_zero}) is obvious since $h(p,p+(\frac{1}{2},\frac{1}{2})) = h(p-(\frac{1}{2},\frac{1}{2}), p)$ if $p\in\mathbb{Z}^2$.

To check the identity (\ref{eq:sharp_zero}), define the bijections $\omega^{+}:Q_{\#}\longrightarrow Q_{\uparrow\rightarrow}$
and $\omega^{-}:Q_{\#}\longrightarrow Q_{\rightarrow\uparrow}$
by
\[
\omega^{\pm}(\path{p},i):=(\omega_{i}^{\pm}(\path{p}),i),
\]
where
\[
\omega_{i}^{\pm}(p_{0},\ldots,p_{n+m+1}):=\left(p_{0},\ldots,p_{i}\pm(\frac{1}{2},-\frac{1}{2}),\ldots,p_{n+m+1}\right).
\]
Then
\[
s(\omega^{+}(\path{p},i))+s(\omega^{-}(\path{p},i))+s(\path{p},i)=C(\path{p},i)\, e_{f(p_{1})}\cdots\widehat{e_{f(p_{i})}}\cdots e_{f(p_{n+m})}
\]
with
\begin{align*}
C(\path{p},i) & = {\rm sgn}(\omega_{i}^{+}(\path{p}))\left(h((x+1,y),(x+1,y+1))-h((x,y),(x+1,y))\right)\\
 &\ \  +{\rm sgn}(\omega_{i}^{-}(\path{p}))\left(h((x,y+1),(x+1,y+1))-h((x,y),(x,y+1))\right)\\
 &\ \  +{\rm sgn}(\path{p})\left(h((x+\frac{1}{2},y+\frac{1}{2}),(x+1,y+1))-h((x,y),(x+\frac{1}{2},y+\frac{1}{2}))\right)\\
 & =-{\rm sgn}(\path{p})\left([b_{y+1}-b_{y}]+[a_{x+1}]-[a_{x+1}-a_{x}]-[b_{y}]\right)\\
 & \ \ -{\rm sgn}(\path{p})\left([a_{x+1}-a_{x}]+[b_{y+1}]-[b_{y+1}-b_{y}]-[a_{x}]\right)\\
 & \ \ +{\rm sgn}(\path{p})\left([a_{x+1}]+[b_{y+1}]-[a_{x}]-[b_{y}]\right)\\
 & =0
\end{align*}
where we put $p_{i-1}=(x,y)$ (see Figure \ref{fig:fig1}).
\begin{figure}[h]
\[
\xymatrix@C=60pt{ &  &  &  & \ \\
 & p_{i}+(\frac{1}{2},-\frac{1}{2})\ar^{[b_{y+1}-b_{y}]+[a_{x+1}]}[rr] &  & p_{i+1}\ar@{~>}[ru]\\
 &  & p_{i}\ar^(.3){[a_{x+1}]+[b_{y+1}]\ \ \ \ }[ru]\\
 & p_{i-1}=(x,y)\ar_{[b_{y+1}-b_{y}]+[a_{x}]}[rr]\ar_(.6){\ \ [a_{x}]+[b_{y}]}[ru]\ar^{[a_{x+1}-a_{x}]+[b_{y}]}[uu] &  & p_{i}-(\frac{1}{2},-\frac{1}{2})\ar_{[a_{x+1}-a_{x}]+[b_{y+1}]}[uu]\\
\ \ar@{~>}[ru]
}
\]
\caption{\label{fig:fig1}The paths $\path{p}$ (diagonal), $\omega_{i}^{+}(\path{p})$ (upper), $\omega_{i}^{-}(\path{p})$ (lower).}
\end{figure}
Hence we obtain the identity (\ref{eq:sharp_zero}).

To check the identity (\ref{eq:ee_der}), we first decompose $Q_{\rightarrow\rightarrow}$,
$Q_{\rightarrow\nearrow}$ and $Q_{\nearrow\rightarrow}$ as
\[
Q_{\alpha,\beta}=\bigsqcup_{y=1}^{m}Q_{\alpha,\beta}^{(y)}
\]
where $(\alpha,\beta)\in\{(\rightarrow,\rightarrow),(\rightarrow,\nearrow),(\nearrow,\rightarrow)\}$
and
\[
Q_{\alpha,\beta}^{(y)}=\{(\path{p},i)\in Q_{\alpha,\beta}\mid \exists x, p_{i}=(x,y)\}.
\]
On the other hand, by definition, we have
\[
u*\partial v=\sum_{y=1}^{m}\left([b_{y+1}-b_{y}]-[b_{y}-b_{y-1}]\right)u*e_{b_{1}}\cdots\widehat{e_{b_{y}}}\cdots e_{b_{m}}
\]
where we set $b_{0}=0$.
By similar methods as for the identity (\ref{eq:sharp_zero}), we can show
\[
s(Q_{\rightarrow\rightarrow}^{(y)})+s(Q_{\rightarrow\nearrow}^{(y)})+s(Q_{\nearrow\rightarrow}^{(y)})=\left([b_{y+1}-b_{y}]-[b_{y}-b_{y-1}]\right)u*e_{b_{1}}\cdots\widehat{e_{b_{y}}}\cdots e_{b_{m}}
\]
(see Figures \ref{fig:RRU} and \ref{fig:RR}). Hence we have the identity (\ref{eq:ee_der}).

\begin{figure}[h]
  \[
  \xymatrix{ &  &  &  & \  &  & \ \\
   &  &  & a_{x+1}b_{y}\ar^{[b_{y+1}-b_{y}]+[a_{x+1}]}[rr] &  & a_{x+1}b_{y+1}\ar@{~>}[ru]\\
   &  & 0\ar^{[a_{x+1}]+[b_{y}]}[ru] &  & 0\ar[ru]\\
   & a_{x}b_{y-1}\ar_{[b_{y}-b_{y-1}]+[a_{x}]}[rr]\ar[ru] &  & a_{x}b_{y}\ar_{[a_{x}]+[b_{y}]}[ru]\\
  \ \ar@{~>}[ru]
  }
  \]
  \[
  \xymatrix{ &  &  &  & \ \\
   & a_{1}b_{1}\ar^{[b_{2}-b_{1}]+[a_{1}]}[rr] &  & a_{1}b_{2}\ar@{~>}[ru]\\
  0\ar^{[a_{1}]+[b_{1}]}[ru]
  }
  \]
\caption{\label{fig:RRU}A visual explanation of the reason why the terms in $s(Q_{\rightarrow\nearrow}^{(y)})+s(Q_{\nearrow\rightarrow}^{(y)})$
corresponds to the terms of the form $\left([b_{y+1}-b_{y}]-[b_{y}-b_{y-1}]\right)\times(\cdots e_{a_{x}b_{y-1}}e_{0}e_{a_{x+1}b_{y+1}}\cdots)$
which appear in $u*\partial v$. The first diagram is for the case $y>1$
and the second diagram is for the case $y=1$. Note that $\left([b_{y+1}-b_{y}]+[a_{x+1}]-[a_{x+1}]-[b_{y}]\right)+\left([a_{x}]+[b_{y}]-[b_{y}-b_{y-1}]-[a_{x}]\right)=[b_{y+1}-b_{y}]-[b_{y}-b_{y-1}]$
for the first diagram, and $[b_{2}-b_{1}]+[a_{1}]-[a_{1}]-[b_{1}]=[b_{2}-b_{1}]-[b_{1}]$
for the second diagram.}
\end{figure}
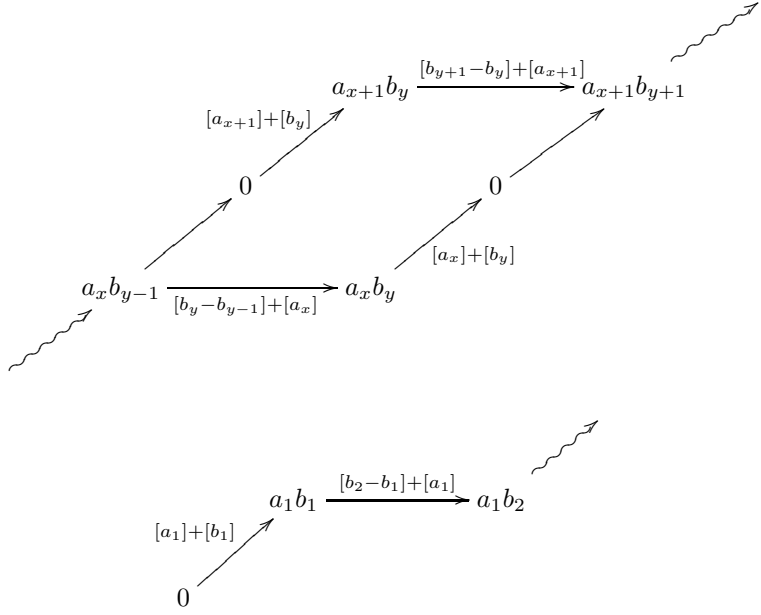

\begin{figure}[h]
\[
\xymatrix@C40pt{ &  &  &  &  &  & \ \\
 & (x,y-1)\ar_-{[b_{y}-b_{y-1}]+[a_{x}]}[rr] &  & (x,y)\ar_-{[b_{y+1}-b_{y}]+[a_{x}]}[rr] &  & (x,y+1)\ar@{~>}[ru]\\
\ \ar@{~>}[ru]
}
\]
\caption{\label{fig:RR}A visual explanation of the reason why the terms in $s(Q_{\rightarrow\rightarrow}^{(y)})$
corresponds to the terms of the form $\left([b_{y+1}-b_{y}]-[b_{y}-b_{y-1}]\right)(\cdots e_{a_{x}b_{y-1}}e_{a_{x}b_{y+1}}\cdots)$
which appear in $u*\partial v$. Note that $[b_{y+1}-b_{y}]+[a_{x}]-[b_{y}-b_{y-1}]-[a_{x}]=[b_{y+1}-b_{y}]-[b_{y}-b_{y-1}]$.}
\end{figure}
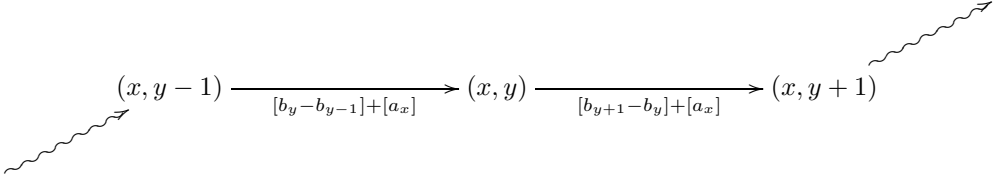
The identity (\ref{eq:nn_der}) is also checked in completely the same way.

Combining the identities (\ref{eq:stpr0}), (\ref{eq:lambda_value}) together with (\ref{eq:ne_ne_zero}), (\ref{eq:sharp_zero}), (\ref{eq:ee_der}) and (\ref{eq:nn_der}), it readily follows that
\[
\partial(u*v)=(\partial u)*v+u*\partial v+\Lambda.\qedhere
\]



\end{proof}

\section{\label{sec:Mbius-transformation}M\"{o}bius transformation}

In this section, we investigate the relation between the differential
operator $\partial^{s,t}$ and the transformation on the algebra $\mathcal{A}=\mathcal{A}_{F}$
associated to the M\"{o}bius transformation on $\mathbb{P}^{1}$.

Recall that ${\rm GL}_{2}(F)$ naturally acts on $\mathbb{P}^{1}(F)=F\sqcup\{\infty\}$
by the M\"{o}bius transformation
\[
\gamma(x)=\frac{ax+b}{cx+d}\ \ \ \left(\gamma=\left(\begin{array}{cc}
a & b\\
c & d
\end{array}\right)\right).
\]
Put $e_{\infty}:=0$. We define an automorphism $\gamma^{*}$ of $\mathcal{A}$
by%
\footnote{The motivation of the definition of $\gamma^{*}$ comes from the formula
$\frac{dt}{t-x}=\frac{ds}{s-\gamma(x)}-\frac{ds}{s-\gamma(\infty)}$,
where $s=\gamma(t)$.%
} $\gamma^{*}(e_{x})=e_{\gamma(x)}-e_{\gamma(\infty)}.$

Hereafter, we fix $\gamma=\left(\begin{array}{cc}
a & b\\
c & d
\end{array}\right)\in{\rm GL}_{2}(F)$, a starting point $s\in F$ and an endpoint $t\in F$. We assume
that $\gamma(s)\neq\infty$ and $\gamma(t)\neq\infty$. We set $\varepsilon_{\gamma}(z):=[\det\gamma]-2[cz+d]\in\mathbb{Z}$.
We shall prove the following identity%
\footnote{The motivation of this identity comes from the total differential
of M\"{o}bius transformation formula for iterated integrals.%
}:
\begin{thm}
\label{thm:lin_fra_main}For a non-constant monomial $w\in\mathcal{A}$,
\[
\left(\gamma^{-1}\right)^{*}\partial^{\gamma(s),\gamma(t)}\gamma^{*}(w)=\partial^{s,t}w+\delta_{x,s}\,\varepsilon_{\gamma}(s)w'-\delta_{y,t}\,\varepsilon_{\gamma}(t)w'',
\]
where $w=e_{x}w'=w''e_{y}$. In particular, for $w\in\mathcal{A}_{(s,t)}^{0}$,
\[
\left(\gamma^{-1}\right)^{*}\partial^{\gamma(s),\gamma(t)}\gamma^{*}(w)=\partial^{s,t}w.
\]

\end{thm}
In the following, we give a proof of this theorem. We put $w=e_{z_{1}}\cdots e_{z_{n}}\in\mathcal{A}_{(s,t)}^{0}$
and $(z_{0},z_{n+1})=(s,t)$. We define a directed graph $G=(V,E)$
whose vertex set $V$ is given by
\[
V=\{(0,0)\}\sqcup\{1,\dots,n\}\times\{0,1\}\sqcup\{(n+1,0)\}
\]
and whose edge set $E$ is given by
\[
E=\{((m_{1},i_{1}),(m_{2},i_{2}))\in V^{2}\mid m_{2}=m_{1}+1\}.
\]
Define a labeling $f$ and a sign ${\rm sgn}:V\to\{\pm1\}$ of vertices
by
\[
f:V\to\mathbb{P}^{1}(F)\ ;\ (n,i)\mapsto\begin{cases}
\gamma(z_{n}) & i=0\\
\gamma(\infty) & i=1,
\end{cases}
\]
and by ${\rm sgn}((m,i))=(-1)^{i}$. For example, if $n=3$, the labeling
and the sign of the graph is as follows:
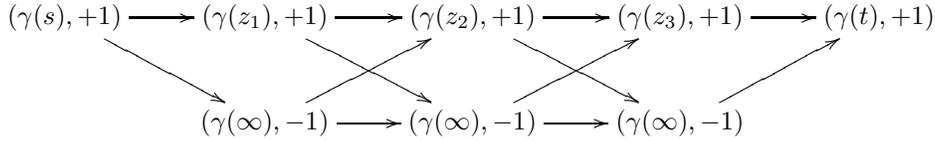
\[
\xymatrix{(\gamma(s),+1)\ar[r]\ar[rd] & (\gamma(z_{1}),+1)\ar[r]\ar[rd] & (\gamma(z_{2}),+1)\ar[r]\ar[rd] & (\gamma(z_{3}),+1)\ar[r] & (\gamma(t),+1)\\
 & (\gamma(\infty),-1)\ar[r]\ar[ru] & (\gamma(\infty),-1)\ar[r]\ar[ru] & (\gamma(\infty),-1)\ar[ru]
}
\]
As before, we define the set of paths by
\[
P:=\left\{ \path{p}=\left(p_{0},\ldots,p_{n+1}\right)\in V^{n+2}\left|\,\begin{array}{c}
p_{0}=\left(0,0\right),\\
p_{n+1}=(n+1,0),
\end{array}\begin{array}{c}
p_{m}\in\left\{ (m,0),(m,1)\right\} \\
\mbox{ for }1\leq m\leq n
\end{array}\right.\right\} ,
\]
the subset of ``effective'' paths by
\begin{align*}
P^{\mathrm{eff}}: & =\left\{ \path{p}\in P\left|\, f(p_{m})\neq\infty\,\mbox{ for }1\leq m\leq n\right.\right\}
\end{align*}
and ${\rm sgn}:P\rightarrow\{\pm1\}$ by $\mathrm{sgn}(\path{p})=\prod_{i=1}^{n}\mathrm{sgn}(p_{i})$.
Then, $\gamma^{*}(w)$ is expressed as
\begin{align*}
\gamma^{*}(w) & =\sum_{\path{p}\in P}{\rm sgn}(\path{p})\, e_{f(p_{1})}\cdots e_{f(p_{n})}=\sum_{\path{p}\in P^{\mathrm{eff}}}{\rm sgn}(\path{p})\, e_{f(p_{1})}\cdots e_{f(p_{n})}.
\end{align*}
Now, define a map $\lambda:E\to\mathbb{Z}$ by
\begin{align*}
 & \lambda(v_{1},v_{2}):=\\
 & \begin{cases}
[ad-bc]+[z_{m+1}-z_{m}]-[cz_{m+1}+d]-[cz_{m}+d] & \mbox{ if }(v_{1},v_{2})=((m,0),(m+1,0))\\
{}[ad-bc]-[c]-[cz_{m}+d] & \mbox{ if }(v_{1},v_{2})=((m,0),(m+1,1))\\
{}[ad-bc]-[c]-[cz_{m+1}+d] & \mbox{ if }(v_{1},v_{2})=((m,1),(m+1,0))\\
{}[ad-bc]-2[c]-[z_{m+1}-z_{m}] & \mbox{ if }(v_{1},v_{2})=((m,1),(m+1,1))
\end{cases}
\end{align*}
for $m\in\{0,\ldots,n\}$. Then, by direct calculations, we have
\begin{equation}
[f(v_{1})-f(v_{2})]=\lambda(v_{1},v_{2})\label{eq:lambda}
\end{equation}
for $(v_{1},v_{2})\in E$ such that $f(v_{1})\neq\infty$, $f(v_{2})\neq\infty$
and $f(v_{1})\neq f(v_{2})$. Set $Q=P\times\{1,\ldots,n\}$ and $Q^{\mathrm{eff}}=P^{\mathrm{eff}}\times\{1,\ldots,n\}$
and
\[
s(\path{p},m):=(\lambda(p_{m},p_{m+1})-\lambda(p_{m-1},p_{m}))(\prod_{\substack{1\leq i\leq n\\
i\neq m
}
}{\rm sgn}(p_{i}))\, e_{f(p_{1})}\cdots\widehat{e_{f(p_{m})}}\cdots e_{f(p_{n})}
\]
for $(\path{p},m)\in Q$. Then, by (\ref{eq:lambda}) and Lemma \ref{lem:lift_ad},
we get
\begin{align*}
\partial^{\gamma(s),\gamma(t)}(\gamma^{*}(w)) & =s(Q^{\mathrm{eff}})+\Lambda,
\end{align*}
where
\begin{align*}
\Lambda:= & \:\delta_{z_{1},s}\,\varepsilon_{\gamma}(s)\sum_{\substack{\path{p}\in P^{{\rm eff}}\\
p_{1}=(1,0)
}
}{\rm sgn}(\path{p})\, e_{f(p_{2})}\cdots e_{f(p_{n})}\\
 & \:-\delta_{z_{n},t}\,\varepsilon_{\gamma}(t)\sum_{\substack{\path{p}\in P^{{\rm eff}}\\
p_{n}=(n,0)
}
}{\rm sgn}(\path{p})\, e_{f(p_{1})}\cdots e_{f(p_{n-1})}\\
= & \:\gamma^{*}\left(\delta_{z_{1},s}\,\varepsilon_{\gamma}(s)w'-\delta_{z_{n},t}\,\varepsilon_{\gamma}(t)w''\right)\ \ \ (w=e_{z_{1}}w'=w''e_{z_{n}}).
\end{align*}
By setting
\begin{eqnarray*}
Q' & = & \left\{ \left.(\path{p},m)\in Q\right|\exists i\neq m\mbox{ such that }f(p_{i})=\infty\right\} ,\\
Q'' & = & \left\{ \left.(\path{p},m)\in Q\right|f(p_{m})=\infty,f(p_{m-1}),f(p_{m+1})\neq\infty\right\} ,
\end{eqnarray*}
we have $Q=Q^{\mathrm{eff}}\cup Q'\cup Q''$. Also, we can check by
direct calculations that
\[
\lambda(v_{1},v_{2})=\lambda(v_{2},v_{3})
\]
for $v_{1},v_{2},v_{3}\in V$ such that $(v_{1},v_{2}),(v_{2},v_{3})\in E$
and $f(v_{1}),f(v_{3})\neq\infty$, $f(v_{2})=\infty$. Thus, $s(\path{p},m)=0$
for $(\path{p},m)\in Q''$. Together with the trivial fact that $e_{f(p_{1})}\cdots\widehat{e_{f(p_{m})}}\cdots e_{f(p_{n})}=0$
for $(\path{p},m)\in Q'$, we find that $s(\path{p},m)=0$ for $(\path{p},m)\in Q'\cup Q''$
and hence
\[
s(Q^{\mathrm{eff}})=s(Q).
\]
Again, by direct calculations,
\begin{align*}
 & \sum_{j\in\{0,1\}}(-1)^{j}\left(\lambda((m,j),(m+1,k))-\lambda((m-1,i),(m,j))\right)\\
 & =[z_{m+1}-z_{m}]-[z_{m}-z_{m-1}]\qquad(i,k\in\{0,1\}).
\end{align*}
Hence, we have
\begin{equation}
s(Q)=\sum_{m=1}^{n}([z_{m+1}-z_{m}]-[z_{m}-z_{m-1}])\sum_{\path{p}\in P_{m}}e_{f(p_{1})}\cdots\widehat{e_{f(p_{m})}}\cdots e_{f(p_{n})}\prod_{\substack{1\leq i\leq n\\
i\neq m
}
}{\rm sgn}(p_{i}),\label{eq:linpr3}
\end{equation}
where $P_{m}$ is a coset of $P$ by the equivalence relation $(p_{0},\ldots,(m,0),\ldots,p_{n+1})\sim(p_{0},\ldots,(m,1),\ldots,p_{n+1})$.
Since the right-hand side of (\ref{eq:linpr3}) is $\gamma^{*}(\partial^{s,t}w)$
by definition, this proves Theorem \ref{thm:lin_fra_main}.

\section{\label{sec:An-application-to01z}An application to iterated integrals
on $\mathbb{P}^{1}\setminus\{0,1,\infty,z\}$}

In this section, we consider the special case $F=\mathbb{Q}(z)$. We
put $\mathcal{A}_{\{0,1,z\}}=\mathbb{Z}\left\langle e_{0},e_{1},e_{z}\right\rangle \subset\mathcal{A}_{F}$
and $\mathcal{A}_{\{0,1\}}=\mathbb{Z}\left\langle e_{0},e_{1}\right\rangle \subset\mathcal{A}_{F}$.
Then, by definition, we have
\begin{align*}
\partial u & \in\mathcal{A}_{\{0,1,z\}}\ \ \ (u\in\mathcal{A}_{\{0,1,z\}}),\\
u\shuffle v & \in\mathcal{A}_{\{0,1,z\}}\ \ \ (u,v\in\mathcal{A}_{\{0,1,z\}}),\\
u*v & \in\mathcal{A}_{\{0,1,z\}}\ \ \ (u\in\mathcal{A}_{\{0,1,z\}},v\in\mathcal{A}_{\{0,1\}}).
\end{align*}
Recall the linear operator $\partial_{z,c}:\mathcal{A}_{F}\to\mathcal{A}_{F}$
introduced in (\ref{eq:der_intro}) i.e.,
\[
\partial_{z,c}(e_{a_{1}}\cdots e_{a_{n}}):=\sum_{i=1}^{n}\left(\delta_{\{a_{i},a_{i+1}\},\{z,c\}}-\delta_{\{a_{i-1},a_{i}\},\{z,c\}}\right)e_{a_{1}}\cdots\widehat{e_{a_{i}}}\cdots e_{a_{n}},
\]
where $(a_{0},a_{n+1})=(0,1)$. Note that $\partial_{z,c}=\partial_{\mathfrak{F}_{c}}^{0,1}$,
where $\mathfrak{F}_{c}:F^{\times}\to\mathbb{Z}$ is any homomorphism
satisfying
\[
\mathfrak{F}_{c}(z-c)=1,\ \mathfrak{F}_{c}(z-(1-c))=0.
\]
Let $\tau_{z}:\mathcal{A}_{\{0,1,z\}}\to\mathcal{A}_{\{0,1,z\}}$
be the anti-automorphism (i.e., $\tau_{z}(uv)=\tau_{z}(v)\tau_{z}(u)$
for $u,v\in\mathcal{A}_{\{0,1,z\}}$) defined by
\[
\tau_{z}(e_{0})=e_{z}-e_{1},\ \tau_{z}(e_{1})=e_{z}-e_{0},\ \tau_{z}(e_{z})=e_{z}.
\]
We also put $\mathcal{A}_{\{0,1\}}^{1}:=\mathcal{A}_{\{0,1\}}\cap\mathcal{A}_{F}^{1}$
and $\mathcal{A}_{\{0,1,z\}}^{0}:=\mathcal{A}_{\{0,1,z\}}\cap\mathcal{A}_{F}^{0}$.
\begin{thm}
\label{thm:Three_applications}For $c\in\{0,1\}$, we have the following
formulas.
\begin{enumerate}
\item For $u,v\in\mathcal{A}_{\{0,1,z\}}$,
\[
\partial_{z,c}(u\shuffle v)=(\partial_{z,c}u)\shuffle v+u\shuffle(\partial_{z,c}v).
\]

\item For non-constant monomials $u\in\mathcal{A}_{\{0,1\}}$ and $v\in\mathcal{A}_{\{0,1,z\}}$,
\[
\partial_{z,c}(u*v)=u*(\partial_{z,c}v)+\delta_{a,0}\delta_{b,z}\delta_{c,0}\, u'*v,
\]
where $u=e_{a}u',v=e_{b}v'$. In particular, for $u\in\mathcal{A}_{\{0,1\}}^{1}$
and $v\in\mathcal{A}_{\{0,1,z\}}$,
\[
\partial_{z,c}(u*v)=u*(\partial_{z,c}v).
\]

\item For a non-constant monomial $u\in\mathcal{A}_{\{0,1,z\}}$,
\[
\tau_{z}^{-1}\circ\partial_{z,c}\circ\tau_{z}(u)=\partial_{z,c}u+(\delta_{c,1}-\delta_{c,0})(\delta_{a,0}u'+\delta_{b,1}u''),
\]
where $u=e_{a}u'=u''e_{b}.$ In particular, for $u\in\mathcal{A}_{\{0,1,z\}}^{0}$,
\[
\tau_{z}^{-1}\circ\partial_{z,c}\circ\tau_{z}(u)=\partial_{z,c}u.
\]

\end{enumerate}
\end{thm}
\begin{proof}
(1) and (2) of Theorem \ref{thm:Three_applications} follow from Theorem
\ref{thm:shuffle} and \ref{thm:stuffle}. Let $\varphi:\mathcal{A}_{F}\to\mathcal{A}_{F}$
be an anti-automorphism defined by
\[
\varphi(e_{x})=-e_{x}\ \ \ (x\in F).
\]
Then, we have
\[
\varphi\circ\partial^{x,y}=\partial^{y,x}\circ\varphi\qquad(x,y\in F).
\]
Since $\tau_{z}=\varphi\circ\gamma_{z}^{*}$ with $\gamma_{z}=\left(\begin{array}{cc}
z & -z\\
1 & -z
\end{array}\right)\in{\rm GL}_{2}(F)$, (3) of Theorem \ref{thm:Three_applications} follows from Theorem
\ref{thm:lin_fra_main}.
\end{proof}

\section*{Acknowledgements}

This work was supported by JSPS KAKENHI Grant Numbers JP18J00982, JP18K13392 and by a Postdoctoral fellowship at the National Center for Theoretical Sciences.

The authors would like to thank Erik Panzer for some useful comments
on a draft of this paper, and Naho Kawasaki for her careful reading and kindly pointing out some errors in the paper.

\bibliographystyle{plain}
\bibliography{memo}

\end{document}